\documentclass[a4paper,11pt,reqno]{amsart}
\usepackage[centertags]{amsmath}
\usepackage{amsfonts,amssymb,amsthm}
\usepackage[dvips]{graphicx} 
\usepackage{subfigure}
\usepackage{psfrag}
\usepackage[english]{babel}
\usepackage{newlfont}
\usepackage{color}
\usepackage[body={16cm,23cm},centering]{geometry}
\usepackage{fancyhdr}
\newtheorem{theorem}{Theorem}
\newtheorem*{mtheorem}{Main Theorem}
\newtheorem{lemma}{Lemma}
\newtheorem{remark}{Remark}

\theoremstyle{definition}

\numberwithin{equation}{section}

\newcommand{\R}{\mathbb{R}}
\newcommand{\Z}{\mathbb{Z}}
\newcommand{\N}{\mathbb{N}}
\renewcommand{\S}{\mathbb{S}}
\newcommand{\Om}{\Omega}
\newcommand{\e}{\varepsilon}
\newcommand{\HH}{\mathcal{H}}

\newcommand{\eff}{\mathrm{eff}}

\newcommand{\dx}{\,\mathrm{d}x}

\newcommand{\MS}{M\hspace{-.1em}S}
\mathchardef\emptyset="001F

\title[Toughening by crack deflection]
{Toughening by crack deflection in the homogenization\\ of brittle composites with soft inclusions}
\author[M. Barchiesi]{M. Barchiesi}
\address[M. Barchiesi]{Dip. di Matematica ed Applicazioni, Universit\`{a} di Napoli ``Federico II'', 
Via Cintia, 80126 Napoli, Italy}
\email{barchies@gmail.com}
\date{November 7, 2016}

\begin{document}
\maketitle

\begin{center}
\begin{minipage}{11.5cm}
\small{
\noindent {\bf Abstract.} 
We present a simple example of toughening mechanism in the homogenization of composites with soft inclusions,
produced by crack deflection at microscopic level. We show that the mechanism is connected to the irreversibility of 
the crack process. Because of that it cannot be detected through the standard homogenization tool of the $\Gamma$-convergence.}
\end{minipage}
\end{center}
\bigskip

\section{introduction}

\noindent
In this paper we focus on the toughening mechanism in composites, and more precisely on that produced by crack-path deflection
at microscopic level. It occurs whenever interactions between the crack and an inclusion cause the crack site to evolve
in the matrix out of the path expected in an homogeneous material. 
In this way, the energy required to open and enlarge a crack in the material increases.

Our aim is to replicate this mechanism within 
the framework of the weak formulation of Griffith's theory of brittle fracture (see~\cite{AB95}) with ad hoc model. 
According to the weak formulation, we assumed that the displacement $u$ belongs to the class of 
\emph{Special functions with Bounded Variation}. Within this functional framework, the crack site
is identified with the set $S_u$ of the discontinuities of $u$, the orientation of the crack is 
described by the normal $\nu_u$ to $S_u$, and the opening of the crack is identified with the jump $[u]$.
At microscopic level, the equilibrium configurations of the system are reached by minimizing the sum $F_\e(u)$ of
the elastic energy stored in the uncracked part of the body, and the surface energy dissipated to open the crack.
The small parameter $\e$ takes into account the size of the heterogeneity of the composite.
It is important to determine, at macroscopic level, the effective material properties, i.e., to replace 
the composite with an ideal homogeneous material. Among the various properties, we are interested in the toughness.
Since the analysis rests on the study of equilibrium states, or minimizers, 
of the energy $F_\e$, it is natural to use the $\Gamma$-convergence as homogenization tool in order 
to describe such an effective property. However, the toughening mechanisms cannot be captured by the 
$\Gamma$-limit $F$ of the family $(F_\e)$, as the size $\e$ of the heterogeneities goes to zero.
On the other hand, our analysis enlightens what is the missing information, i.e., that the mechanism is originated 
by the \emph{irreversibility of the crack process}. When we add this constraint, then the asymptotic behavior
of the family $(F_\e)$ reproduces the toughening.

Our model is very simple: it is a composite constituted by a brittle matrix with soft inclusions 
arranged at microscopic level in a sort of chessboard structure (see Figure~\ref{fig1}).
The matrix and the soft inclusions have different elastic moduli, but the same toughness, 
here normalized to one. 

At the microscale the surface energy part of $F_\e$ depends only on the length of the crack $\HH^1(S_u)$, while 
at the macroscale the $\Gamma$-limit $F$ is of cohesive type, in the meaning that the surface energy depends 
also on the opening of the crack:
\begin{equation*}
\int_{S_u}g([u],\nu_u)\,\mathrm{d}\HH^1
\end{equation*}
for a certain surface energy density $g$. 
From the physical point of view, the fracture energy is not completely dissipated at crack iniziation but, 
due to the interaction between the crack's faces, also during the opening of the crack.

Because of that the surface energy associate to different displacements with 
the same crack site could be different. Assume for instance to have a piecewise constant displacement $u_t$
having a horizontal crack, i.e., $\nu_u$ is constantly equal to the vertical direction $e$, and with opening $[u]=t$ 
between the crack's faces. 
How is $F(u)$ determined? From the operative point of view, we have to  build a family of displacements $(u_\e)$
converging to $u$ on one side, and minimizing the energy $F_\e$ on the other one. Then $F(u)$ will be the limit
of $(F_\e(u_\e))$ as $\e$ goes to zero. Now, when $t$ is small, at the microscopic level the soft inclusions can
be stretched paying a few amount of bulk energy also in case of high gradients. In a certain sense, from the energetic
point of view, the material behaves as if there are perforations in place of soft inclusion. Because of that, the best way
to approximate $u$ is with a zig-zag configuration (see Figures \ref{fig2}-\ref{fig3-bis}), 
i.e., a displacement $u_\e$ having crack site going from a soft inclusion to another one (so, in particular, not horizontal) and 
that stretches the soft regions without breaking them (a sort of bridges between the two opposite faces of the macroscopic crack). 
In particular this shows that the limit model $F$ has a positive activation threshold $g(0^+,e)=1/\sqrt{2}$ strictly smaller 
than that of $F_\e$, that it is one (as expected, because the presence of the soft inclusions). 
On the other hand, when $t'$ is large, it is no longer energetically convenient to stretch the soft inclusions instead of 
breaking them. Because of that the best way to approximate $u_{t'}$ is with $u_{t'}$ itself. Indeed $g(t',e)=1$ for $t'$ larger than
a certain threshold. So, the $\Gamma$-limit does no detect any increment of the toughness!
However, the point is that if $u_{t'}$ is an evolution of $u_t$, then in building the approximation $(v_\e)$ for $u_{t'}$ 
we have to take into account the irreversibility of the crack process, i.e., at every fixed $\e$ the crack site of $v_\e$
has to contain the crack site of $u_\e$:
\begin{equation*}
S_{v_\e}\supset S_{u_\e}.
\end{equation*}
If we add this constraint, then we cannot set $v_\e=u_{t'}$ but we can only modify the previous sequence $u_\e$ by keeping
the zig-zag configuration in the matrix and extending the crack inside the soft inclusions.  
In this way we obtain an effective surface energy density $g_{\eff}$ such that $g_{\eff}(t',e)=1/2+1/\sqrt{2}>1=g(t',e)$ 
for $t'$ large: there is an increment in resistance to large crack-opening and to further growth.

\section{Setting of the problem and presentation of the results}

\noindent
Let $\Om$ be an open bounded subset of $\R^2$. The space of special functions of bounded variation on $\Om$ 
will be denoted by $SBV(\Om)$. For the general theory we refer to \cite{Amb00}. 
For every $u\in SBV(\Om)$, $\nabla u$ denotes the \emph{approximate gradient} of $u$, 
$S_u$ the \emph{approximate discontinuity set} of $u$ (the crack site), and $\nu_u$ the \emph{generalized normal} 
to $S_u$, which is defined up to the sign. If $u^+$ and $u^-$ are the traces of $u$ on the sides of $S_u$ determined 
by $\nu_u$, the difference $u^+-u^-$ is called the \emph{jump} of~$u$ (the opening of the crack) and is denoted by~$[u]$. 

Our ambient space is the subspace of $SBV(\Om)$ given by
\begin{equation*}
SBV^2(\Om):=\{u\in SBV(\Om)\colon \nabla u\in L^2(\Om;\R^2) \text{ and } \HH^1(S_u)<+\infty\}.
\end{equation*}
We consider also the larger space of \emph{generalized special functions of bounded variation} on $\Om$, $GSBV(\Om)$, 
which is made of all the integrable functions $u:\Om\rightarrow\R$ whose truncations $u^m:=(u\wedge m)\vee (-m)$ belong 
to $SBV(\Om)$ for every $m\in \N$. 
In analogy with the case of $SBV$ functions, we say that $u\in GSBV^2(\Om)$ if $u\in GSBV(\Om)$, $\nabla u \in L^2(\Om;\R^2)$ 
and $\HH^1(S_u)<+\infty$. 

We also recall the definition of the Mumford-Shah functional 
\begin{equation*}
\MS(u,\Om):=\begin{cases}
\displaystyle\int_\Om |\nabla u|^2\dx+\HH^1(S_u) & \text{if } u\in GSBV^2(\Om),
\cr
+\infty & \text{otherwise in } L^1(\Om).
\end{cases}
\end{equation*}
By \cite[Theorem 4.36]{Amb00}, the $\MS$ functional is $L^1$-lower semicontinuous on $GSBV^2(\Om)$.

For $r>0$ we denote by $Q_r$ the square with side-length $r$, centered at the origin, i.e.,  
$Q_r:=(-r/2,r/2)^2$; while we simply write $Q$ in place of $Q_1$.
Finally, we indicate by $u_t$ the function on $Q$ defined by
\begin{equation*}
u_t:=t\chi_{(-1/2,1/2)\times(0,1/2)}.
\end{equation*}

\medskip

In what follows the $\Gamma$-convergence of functionals is always understood with respect to the strong $L^1$-topology. 
For the general theory about $\Gamma$-convergence we refer to the short presentation in \cite{B06} and the references therein.

\medskip

Let us introduce our model. We set (see Figure \ref{fig1})
\begin{equation*}\begin{split}
D&:=\overline Q_{\frac14} \cup \Bigl( \overline Q_{\frac18} \pm \bigl(\tfrac{7}{16},\tfrac{7}{16}\bigr) \Bigr) \cup \Bigl( \overline Q_{\frac18} \pm \bigl(\tfrac{7}{16},-\tfrac{7}{16}\bigr) \Bigr),\\
P&:=\R^2 \setminus \bigcup_{i\in\Z^2} (D+i).
\end{split}\end{equation*}

\begin{figure}
\centering
\includegraphics[width=11cm]{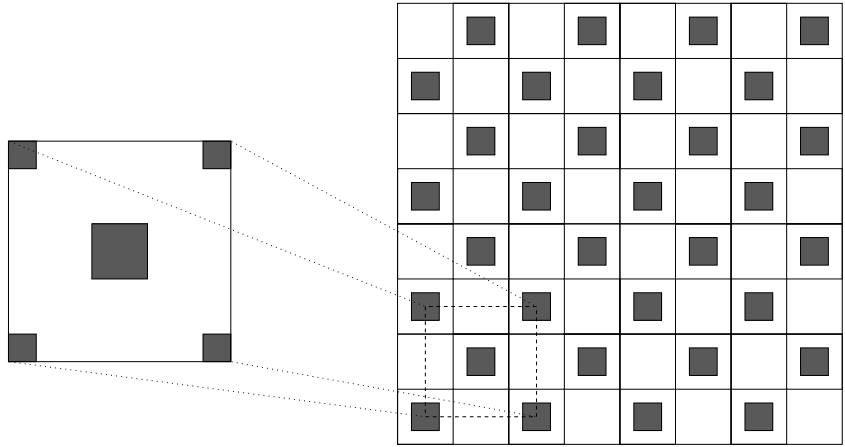}
\caption{In gray, the sets $D$ (on the left) and $\R^2\setminus P$ (on the right).}
\label{fig1}
\end{figure}

We are interested in the asymptotic behavior of the 
functionals $F_\e:L^1(\Omega)\to[0,+\infty]$ defined as
\begin{equation}\label{F-eps}
F_\e(u,\Om):=
\begin{cases}
\displaystyle{\int_{\Om \cap \e P}|\nabla u|^2\dx+\e\int_{\Om\setminus\e P}|\nabla u|^2\dx+\HH^1(S_u)} 
& \text{if }\ u\in SBV^2(\Om), \cr
+\infty & \text{otherwise in } L^1(\Om).
\end{cases}
\end{equation}
In the setting of linearized elasticity and antiplane shear, $\Om$ represents the cross section of a cylindrical body
in its reference configuration, while $F_\e(u,\Om)$ represents the energy corresponding to a displacement $u\colon \Om \to \R$.
The body is a periodic brittle composite made of two constituents having different elastic properties. 
The constituent located in $\Om\setminus\e P$ has elastic modulus represented by the vanishing 
sequence $\e$. For this reason, in what follows, $\Om\setminus\e P$ is referred as the \emph{soft inclusions}. 

\begin{figure}
\centering
\psfrag{1}{$p_1$}
\psfrag{2}{$p_2$}
\psfrag{3}{$p_3$}
\psfrag{4}{$p_4$}
\includegraphics[width=11cm]{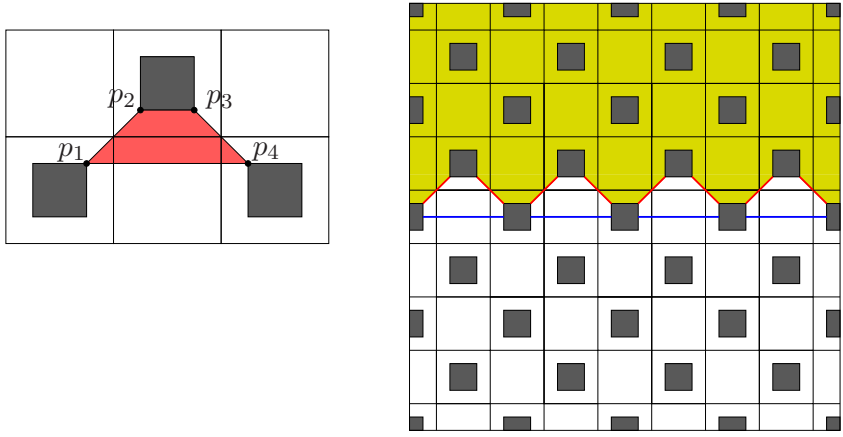}
\caption{In red, the trapezoid $T$ (on the left) and the ``zig-zag'' configuration (on the right).}
\label{fig2}
\end{figure}

We also consider the functionals $\hat F_\e:L^1(\Omega)\to[0,+\infty]$ given by
\begin{equation*}
\hat{F}_\e(u,\Om):=
\begin{cases}
\displaystyle{\int_{\Om \cap \e P}|\nabla u|^2\dx+\HH^1(S_u\cap\Om \cap \e P)}
 & \text{if }u|_{\Om\cap \e P}\in SBV^2(\Om\cap \e P), \\
+\infty &\text{otherwise in }L^1(\Om).
\end{cases}
\end{equation*}
From the physical point of view $\Om\setminus\e P$ represents a \emph{perforation}. 
The asymptotic behavior of functionals like $\hat F_\e$ has been extensively studied in \cite{Bar10,CS07,FGP07}.
Specifically, it has been shown that $(\hat{F}_\e)$ $\Gamma$-converges to
\begin{equation*}
\hat{F}(u,\Om):=
\begin{cases}
\displaystyle{\int_{\Om}f(\nabla u)\dx+\int_{S_u}\hat{g}(\nu_u)\,\mathrm{d}\HH^1} & \text{if }u\in GSBV^2(\Om), \\
+\infty &\text{otherwise in }L^1(\Om),
\end{cases}
\end{equation*}
where $f:\R^2\to[0,+\infty)$ and $\hat{g}:\S^1\to[0,+\infty)$ satisfy
\begin{equation}\begin{split}\label{homcoerc}
c_1|\xi|^2\leq f(\xi)\leq |\xi|^2
\quad&\text{ for every } \xi\in\R^2,\\
c_2\leq \hat{g}(\nu)\leq 1
\hspace{27pt}&\text{ for every } \nu\in\S^1,
\end{split}\end{equation}
for some constants $c_1,c_2{>}0$ only depending on $P$. Moreover, denoted by $e$ the (unitary) vertical vector,
the homogenization formula for $\hat g$ (see \cite[Theorem 4]{Bar10}) gives $\hat{g}(e)=1/\sqrt{2}$. Indeed, 
\begin{equation}\label{cell formula}
\hat{g}(e)=\lim_{\e\rightarrow 0^+}\inf\bigl\{\HH^1(S_{w}) \colon w\in SBV_{0,1}(Q\cap\e P) \bigr\},
\end{equation}
where $SBV_{0,1}(Q\cap\e P)$ is the family of functions $w\in SBV(Q\cap\e P)$ such that $\nabla w=0$ a.e.
in $Q\cap\e P$, with $w(x)=1$ [respect. $0$] on a neighborhood of $\partial Q\cap \{x_2\ge0\}$
[respect. $\partial Q\cap \{x_2<0\}$].
The functions $w\in SBV_{0,1}(Q\cap\e P)$ having the shortest discontinuity set are those such that $S_w$
connects in diagonals two close perforations. Among all the possible configurations, let us consider
the simplest one. We define the trapezoid $T$ of vertices 
$p_1:=(1/8,1/8), p_2:=(3/8,3/8), p_3:=(3/8,5/8), p_4:=(7/8,1/8)$ (see Figure \ref{fig2}), and the sets 
\begin{equation}\begin{split}\label{zig-zag}
Z&:=[0,1)\times[\tfrac18,+\infty)\setminus T,\\
Z_\e&:=Q\cap \e\bigcup_{i\in\Z}\bigl(Z+(i,0)\bigr).
\end{split}\end{equation}
Then, if $\e^{-1}$ is an integer, the function $w=\chi_{Z_\e}$ belongs to $SBV_{0,1}(Q\cap\e P)$. 
Its discontinuity set is the the ``zig-zag'' configuration in Figure \ref{fig2} (in red) and $\HH^1(S_w)=1/\sqrt{2}$.
If $\e^{-1}$ is not an integer, it is enough to slightly modify $w$ in a neighborhood of the points $(-1/2,0)$ and $(1/2,0)$,
possibly increasing $S_w$ of a quantity vanishing as $\e$.

If instead we take $w=\chi_{(-1/2,1/2)\times(0,1/2)}$, then $S_w$ is as in Figure \ref{fig2} (in blue) and $\HH^1(S_w)=3/4$.
Therefore, if the discontinuity set is horizontal, then it is longer. This will be \emph{crucial} in our analysis. 

\medskip

The asymptotic behavior of functionals like $F_\e$ has been instead studied only more recently in \cite{BLZ}. 
Specifically, it can be shown that for $(F_\e)$ the following result holds. Since the microgeometry 
considered in \cite{BLZ} is slightly different from the one considered here, we give a short proof highlighting 
the steps that differ from the original one.

\begin{theorem}
For every decreasing sequence of positive numbers converging to zero, there exists a subsequence $(\e_{k})$ 
such that $(F_{\e_{k}})$ $\Gamma$-converges to a functional 
$F\colon L^1(\Om) \rightarrow [0,+\infty]$ of the form
\begin{equation*}
F(u,\Om):=
\begin{cases}
\displaystyle{\int_{\Om}f(\nabla u)\dx+\int_{S_u}g([u],\nu_u)\,\mathrm{d}\HH^1} & \text{if } u\in GSBV^2(\Om), \cr
+\infty & \text{otherwise in } L^1(\Om),
\end{cases}
\end{equation*}
where $f$ is as in \eqref{homcoerc} and $g:\R\times\S^1\to[0,+\infty)$ is a Borel function satisfying the following properties:
\begin{itemize}
\item[(i)] for every $t\neq0$ and $\nu\in \mathbb S^1$
\begin{equation*}
\hat{g}(\nu)\leq  g(t,\nu)\leq 1;
\end{equation*}
\item[(ii)] for any fixed $\nu \in \mathbb S^1$, $g(\cdot,\nu)$ is nondecreasing and left-continuous in $(0,+\infty)$ 
and satisfies the symmetry condition $g(-t,-\nu)=g(t,\nu)$;
\smallskip
\item[(iii)] for every $t>0$,  $g(\cdot,e)$ satisfies the estimate from above
\begin{equation}\label{crescita1}
g(t,e)\leq\tfrac1{\sqrt{2}}+2\sqrt{2}\,t.
\end{equation}
In particular, $g({0^+},e)=\hat{g}(e)=1/\sqrt{2}$. Moreover, there exists a threshold $t_0>0$ such that
\begin{equation} \label{crescita2}
g(t,e)=1  \quad \text{for} \ t\ge t_0 .
\end{equation}
\end{itemize}
\end{theorem}

\begin{figure}
\centering
\includegraphics[width=11cm]{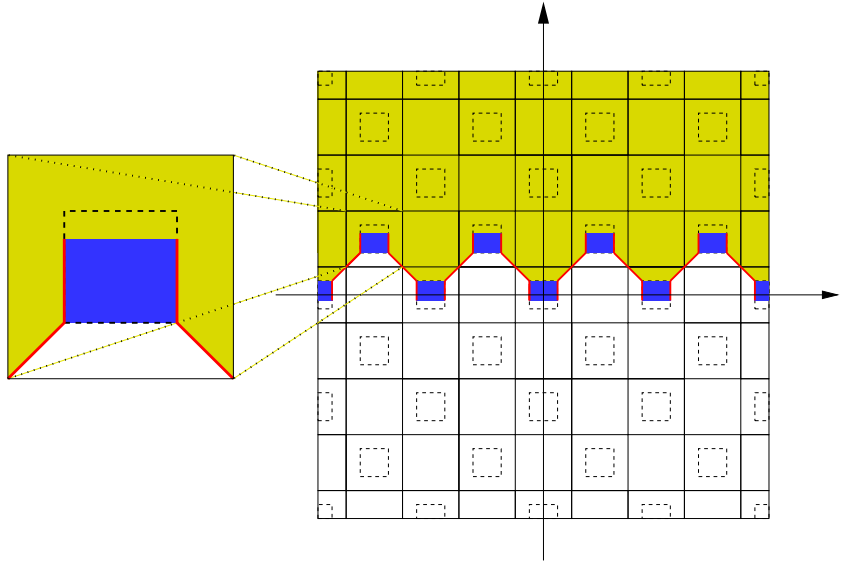}
\caption{In yellow the set $Z_k\setminus(R_k\cup R'_k)$ where $u_k$ takes value $t$, in blue the set $R_k\cup R'_k$ 
where $u_k$ is affine, and in red the discontinuity set $S_{u_k}$.}
\label{fig3}
\end{figure}

\begin{figure}
\centering
\includegraphics[width=14cm]{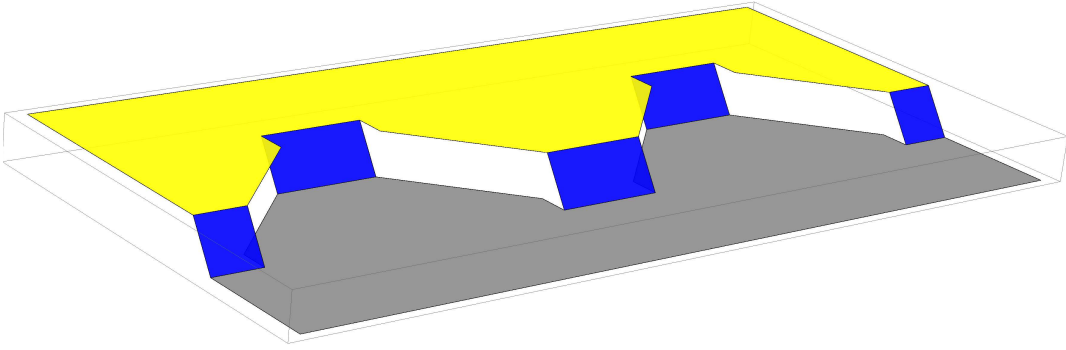}
\caption{the profile of the function $u_k$.}
\label{fig3-bis}
\end{figure}
\begin{proof}
The integral representation of the $\Gamma$-limit and points (i) and (ii) follow as a particular case of \cite[Theorem 1]{BLZ}. 
We divide the proof of (iii) into two steps: one for \eqref{crescita1} and one for~\eqref{crescita2}.

\medskip

\noindent \textbf{Estimate \eqref{crescita1}.}
Since $g(\cdot,e)\leq1$, it is enough to show that 
$g(t,e)\leq 1/\sqrt{2}+2\sqrt{2}\,t$ whenever $t\leq(\sqrt{2}-1)/4$. 
To this aim, consider the sets 
\begin{equation*}\begin{split}
R &:=(-\tfrac18,\tfrac18)\times(\tfrac18-\tfrac{t}{\sqrt{2}},\tfrac18)\\
\quad R'&:=(\tfrac38,\tfrac58)\times(\tfrac38,\tfrac38+\tfrac{t}{\sqrt{2}})
\end{split}\end{equation*}
and 
\begin{equation*}\begin{split}
R_\e&:=Q\cap \e\bigcup_{i\in\Z}\bigl(R+(i,0)\bigr)\\
R'_\e&:=Q\cap \e\bigcup_{i\in\Z}\bigl(R'+(i,0)\bigr).
\end{split}\end{equation*}
Then, with $Z_\e$ as in \eqref{zig-zag}, let $(u_k)\subset SBV^2(Q)$ be the sequence of ``bridging'' functions defined as 
\begin{equation*}
u_k(x):=
\begin{cases}
t & \textrm{if } x\in Z_{\e_k}\setminus(R_{\e_k}\cup R'_{\e_k}),\\[2pt]
t-\frac{\sqrt{2}}{8}+\frac{\sqrt{2}}{\e_k}\,x_2 & \textrm{if } x\in R_{\e_k},\\[2pt]
-\frac{3\sqrt{2}}{8}+\frac{\sqrt{2}}{\e_k}\,x_2 & \textrm{if } x\in R'_{\e_k},\\[2pt]
0 & \textrm{if } x\in Q\setminus(Z_{\e_k}\cup R_{\e_k} \cup R'_{\e_k}),
\end{cases}
\end{equation*}
(see Figures~\ref{fig3} and \ref{fig3-bis}). 
Note that with the choice of $t$ we have $R\subset Q_{1/4}$ and $R'\subset Q_{1/4}+(1/2,1/2)$,
and therefore $R_{\e}\cup R'_{\e}\subset Q\setminus\e P$.
We clearly have $u_k\rightarrow u_t$ in $L^1(Q)$; moreover
\begin{equation*}
\int_{R_{\e_k}}|\nabla u_k|^2\dx\leq(\lfloor\tfrac1{\e_k}\rfloor+1)\sqrt{2}\,t\quad 
\text{and}\quad \HH^1(S_{u_k})\leq\e_k(\lfloor\tfrac1{\e_k}\rfloor+1)\bigl(\tfrac1{\sqrt{2}}+\sqrt{2}\,t\bigr).
\end{equation*}
Thus we readily deduce
\begin{equation*}
g(t,e)=F(u_t,Q) \leq \limsup_{k\to +\infty} F_{\e_k}(u_k,Q)
\leq \tfrac1{\sqrt{2}}+2\sqrt{2}\,t,
\end{equation*}
and hence the estimate from above.

\medskip

\noindent \textbf{Equality \eqref{crescita2}.}
Let 
\begin{equation*}
\tilde P := \R^2 \setminus \tfrac12 \Bigl( \bigcup_{i\in\Z} \overline Q_{\frac12} + i \Bigr) 
\end{equation*}
and define $\tilde F_\e$ as $F_\e$ 
with $P$ replaced by $\tilde P$, cf.\ \eqref{F-eps}.
Moreover, let $\tilde F$ be the $\Gamma$-limit of $(\tilde F_{\e_k})$ and $\tilde g$ its surface energy density.
Then we can apply (a straightforward modification of) \cite[Theorem 2]{BLZ} to $\tilde F_\e$, obtaining that 
$\tilde g = 1$ for $t$ larger than a threshold $t_0$.
On the other hand, since $\tilde P\subset P$, we have $\tilde F_\e \le F_\e$, which implies $\tilde F \le F$ 
and by locality $\tilde g \le g$.
\end{proof}

\vspace{8pt}
In what follows we fix a decreasing sequence $(\e_k)$ of positive numbers such that $(F_{\e_k})$ $\Gamma$-converges to 
a functional $F$. Just in order to simplify the construction involved in our results, we assume that $(\e_k^{-1})$ is
is a sequence of odd integers. In this way for any fixed $k$ the unitary cell $Q$ is divided precisely in 
periodicity cells of side $\e_k$, one centered in the origin. 

\vspace{8pt}
We are mainly interested in the local minima of the $\Gamma$-limit $F$. Fixed $t>0$ and given $\delta>0$, denote by $w_t$ 
a solution to the problem
\begin{equation}\label{min prob}
\begin{cases}
&\min F(w,Q) \colon w\in SBV^2(Q),\\
& w= 0 \,\text{ in }  (-1/2,1/2)\times(-1/2,-\delta/2),\\
& w= t \,\,\text{ in }  (-1/2,1/2)\times(\delta/2,1/2).
\end{cases}
\end{equation}

\begin{lemma}
For any given $t>0$ and given $\delta>0$, there exists a solution $w_t$ to the minimum problem \eqref{min prob}
constant in the horizontal direction, i.e.,
\begin{equation}\label{constant}
w_t(x_1,x_2)=\hat{w}_t(x_2)
\end{equation}
for a certain $\hat{w}_t\in SBV^2((-1/2,1/2))$.
\end{lemma}
\begin{proof}
Since the energy decreases by truncation, in searching for solution to \eqref{min prob} we can always assume the additional
$L^\infty$-bound $w(x)\in[0,t]$. Then, the compactness in $SBV$ and the direct method of calculus of variations
provide the existence of a solution $w$ to \eqref{min prob}. For any $j\in\N$ and $i\in\{1,\ldots,j\}$ 
let $S_j^i$ be the open strip $(-1/2+(i-1)/j,-1/2+i/j)\times(-1/2,1/2)$. Moreover, let $i_j\in\{1,\ldots,j\}$ be 
a solution to the problem
\begin{equation*}
\min\bigl\{F(w,S_j^i) \colon i=1,\ldots,j\bigr\}.
\end{equation*}
We restrict $w$ to $S_j^{i_j}$, and then we extend it to $Q$ by reflection with respect to the axes $x_1=-1/2+i/j$, $i=1,\ldots,j-1$;
we denote by $v_j$ such an extension. 
Because the symmetry of the set $P$, $g(s,(\nu_1,\nu_2))=g(s,(-\nu_1,\nu_2))$ for any $s\in\R$ and $\nu\in\S^1$.
Therefore, we have $F(w,Q)=F(v_j,Q)$ and $v_j$ is still a solution to \eqref{min prob}.
Again by compactness in $SBV$, up to a subsequence $v_j$ converges to a certain $v$. By the lower semicontinuity
of the functional $F$, $v$ is still a solution to \eqref{min prob}. Moreover, since any $v_j$ is $2/j$-periodic in
the variable $x_1$, the function $v$ depends only on $x_2$ and it is the desired solution.
\end{proof}

In what follows we will work with solutions to \eqref{min prob} satisfying condition \eqref{constant}, because
they are easier to handle. Being $f$ a quadratic form, if $\hat{w}_t\in H^1((-1/2,1/2))$, then $\hat{w}_t$
has to be affine in $(-\delta/2,\delta/2)$. 
Noted that the function $u_t:=t\chi_{(-1/2,1/2)\times(0,1/2)}$ has energy $F(u_t,Q)=g(t,e)\leq1$,
we deduce that for $\delta$ small enough $\hat{w}_t$ presents a discontinuity, since the energy of an affine function
blows-up as $\delta$ goes to zero. 
Since $g$ varies between $1/\sqrt{2}$ and $1$, $\hat{w}_t$ cannot have more than one discontinuity point, otherwise
\begin{equation*}
F(w_t,Q)\geq\sqrt{2}>F(u_t,Q).
\end{equation*}
Finally, again by minimality, if $\hat{x}_2$ is the discontinuity point of $\hat{w}_t$, we have that 
$\hat{w}_t$ in affine in $(-\delta/2,\hat{x}_2)$ and $(\hat{x}_2,\delta/2)$. The slopes of the function
in this two intervals depend on $f$, $g$, $t$, and $\delta$. Note that, being $g=g(t,e)$ definitively equal to $1$ 
for $t$ large, beyond a certain threshold $t_0$ it is not energetically favorable not to be flat in $(-\delta/2,\hat{x}_2)$ 
and $(\hat{x}_2,\delta/2)$, since the increment of bulk energy is not compensated by the reduction of the surface energy. 
Therefore, $\hat{w}_t=t\chi_{(\hat{x}_2,1/2)}$ for $t$ larger than $t_0$.

Let us now fix a small quantity $\eta>0$ that we will use later. 
Since $g(t,e)\leq1/\sqrt{2}+2\sqrt{2}\,t$, there exists $t=t(\eta)$ such that
\begin{equation}\label{vicino}
F(u_t,Q)=g(t,e)< \tfrac{1}{\sqrt{2}}+\eta.
\end{equation}
For what we said before, we can also choose $\delta=\delta(t)$ so small that any solution $w_t$ 
to the problem \eqref{min prob}-\eqref{constant} has a horizontal discontinuity set $(-1/2,1/2)\times\{\hat{x}_2\}$. 
Note that the solution is not unique, since the point $\hat{x}_2$ can vary. 
Now that $t$ and $\delta$ are fixed as functions of $\eta$, let us also fix a solution $w=w_t$ of the problem \eqref{min prob},
and consider a recovery sequence $(w_k)$ for $w$ with respect to $F_{\e_k}$.
By definition, $F_{\e_k}(w_k,Q)\rightarrow F(w,Q)$ and $w_k\rightarrow w$ strongly in $L^1$
(and therefore weakly in $SBV$) as $k$ goes to infinity.

Our main result is to show where the discontinuity set of $w_k$ concentrates for $k$ going to infinity. 
Let us introduce some other sets in order the better explain the geometrical setting. 

\begin{figure}
\centering
\psfrag{1}{$\varrho$}
\includegraphics[width=13cm]{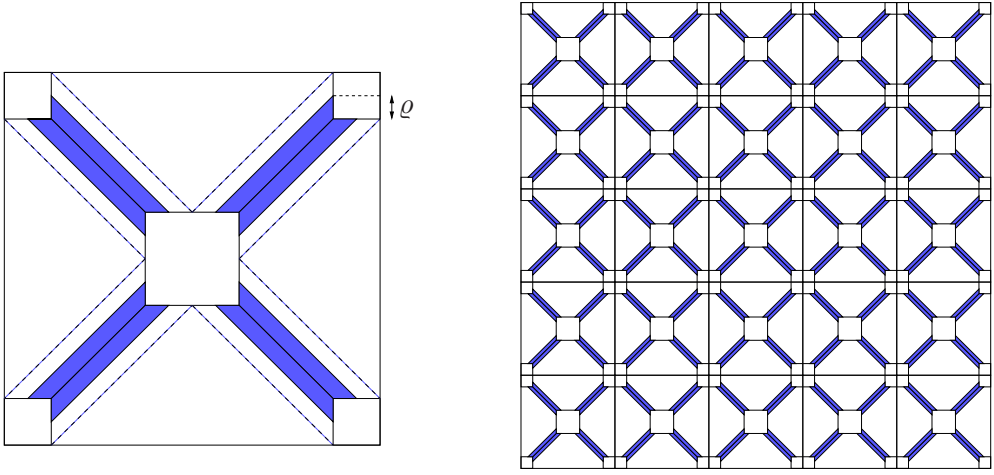}
\caption{In blue the set $T$ (on the left) and the set $T_{\e_k}$ (on the right).}
\label{fig4}
\end{figure}

We fix another small quantity $\varrho\in(0,1/8)$. We denote by $T^1$ the set given by the union of two isosceles 
trapezoids sharing the same short base constituted by the segment $R^1$ of endpoints $(1/8,1/8)$ and $(3/8,3/8)$. 
They have long base constituted respectively by the segment of endpoints $(1/8,1/8-\varrho)$ and $(3/8+\varrho,3/8)$, 
and the segment of endpoints $(1/8-\varrho,1/8)$ and $(3/8,3/8+\varrho)$.
We set $T^3:=-T^1$, while we denote by $T^2$ [respect. $T^4$] the reflection of $T^1$ with respect to the axis $\{x_1=0\}$
[respect. $\{x_2=0\}$]. Finally, we set $T:=\bigcup_{h=1}^4 T^h$ (see Figure \ref{fig4}) and 
\begin{equation}\label{farfalle}
T_\e:=\e\bigcup_{i\in\Z^2}(T+i).
\end{equation}

\begin{theorem}
Given $\eta>0$, choose $t>0$ small enough so that \eqref{vicino} is verified, and $\delta>0$ small enough so
that the solutions to the problem \eqref{min prob}-\eqref{constant} are not affine and discontinuous. Let $w$ be one of this 
solutions, with discontinuity set $(-1/2,1/2)\times\{\hat{x}_2\}$ for a certain $\hat{x}_2\in[-\delta/2,\delta/2]$, 
and $(w_k)$ one of its recovery sequence. Then, given $\varrho\in(0,1/7)$ and defined $T_\e$ as in \eqref{farfalle}, 
\begin{equation}\label{main estimate}
\liminf_{k\rightarrow+\infty}\HH^1(S_{w_k}\cap T_{\e_k})\geq\frac{1}{\sqrt{2}}-\frac{\eta}{4\varrho}.
\end{equation}
\end{theorem}
\begin{proof}
\emph{The basic idea is that for $t$ small the behaviors of $F_{\e_k}$ and  $\hat{F}_{\e_k}$ are similar}.
For the $\Gamma$-limit $\hat{F}$, the cell formula \eqref{cell formula} suggests that the functions of 
a recovery sequence for $w$ should have the discontinuity set concentrate in $T_\e$. Indeed, this is best way to 
cut the hard region $\e P$ (since it is thinner in the diagonals between two close soft inclusions).
In order to simplify the description of the proof, we assume $\hat{x}_2=0$.

First of all, let us define for each $m\in M:=(-1/8,-1/2)\cup(1/8,1/2)$ the fiber 
$l^m$ passing through $m$. If $m\in(1/8,1/2)$, consider the points (see Figure \ref{fig5})
\begin{equation*}\begin{split}
&p_1:=(m,0), \; p_2:=(m,m-1/8), \; p_3:=(1/24+2m/3,1/24+2m/3),\\ 
&p_4:=(m-1/8,m), \; p_5:=(m-1/8,1/2).
\end{split}\end{equation*}
The point $p_3$ belongs to the segment $R^1$ of endpoints $(1/8,1/8)$ and $(3/8,3/8)$, while $p_2$ belongs to the segment 
$S^1$ of endpoints $(0,1/8)$ and $(1/2,3/8)$. The middle point of $R^1$ is $p_7:=(1/4,1/4)$, while the middle point
of $S^1$ is $p_6:=(5/16,3/16)$. The ratio between the distance of $p_3$ from $p_7$ and the distance of $p_2$
from $p_6$ is $2/3$, i.e., the same ration between the lengths of $R^1$ and $S^1$.

We define $\tilde{l}^m$ as the union of the segments of endpoints $(p_1,p_2)$, $(p_2,p_3)$,..., $(p_4,p_5)$, and then
\begin{equation*}
l^m:=\tilde{l}^m\cup\{(x_1,x_2) \colon (x_1,-x_2)\in\tilde{l}^m\}.
\end{equation*}
If $m\in(-1/8,-1/2)$, we define $l^m$ via reflection:
\begin{equation*}
l^m:=\{(x_1,x_2) \colon (-x_1,x_2)\in l^{-m}\}.
\end{equation*}
Note that $Q\setminus D\overset{a.e.}{=}\bigcup\{l^m \colon m\in M\}$ and that the bundle of fibers
undergo a sort of compression of ratio $2/3$ in passing from $S^1\cup S^2$ to $R^1\cup R^2$,
where $R^2$ [respect. $S^2$] is the reflection of $R^1$ [respect. $S^1$] with respect to
the axis $\{x_1=0\}$. We also set $R^3:=-R^1$, while we denote by $R^4$ the reflection of $R^3$ 
with respect to the axis $\{x_2=0\}$, and by $R:=\bigcup_{h=1}^4 R^h$ the union. 

Finally, let us define the periodic and rescaled versions of the sets above:
\begin{equation}\label{nervatura}
M_\e:=\e\bigcup_{i\in\Z}(M+i), \quad R_\e:=\e\bigcup_{i\in\Z^2}(R+i),
\end{equation}
and for $m\in M_\e$
\begin{equation*}
l^m_\e:=\e\bigcup_{i\in\Z}(l^{\frac{m}{\e}-[\frac{m}{\e}]}+i).
\end{equation*}
Note that $\e P\overset{a.e.}{=}\bigcup\{l^m_\e \colon m\in M_\e\}$. In the next two steps we will
show that $S_{w_k}$ intersects asymptotically any fiber $l^m_{\e_k}$, $m\in M_{\e_k}\cap(-1/2,1/2)$,
and that such an intersection takes place mainly close to $Q\cap R_{\e_k}$, in the region $Q\cap T_{\e_k}$.

\begin{figure}
\centering
\psfrag{1}{$p_1$}
\psfrag{2}{$p_2$}
\psfrag{3}{$p_3$}
\psfrag{4}{$p_4$}
\psfrag{5}{$p_5$}
\psfrag{6}{$p_6$}
\psfrag{7}{$p_7$}
\includegraphics[width=7cm]{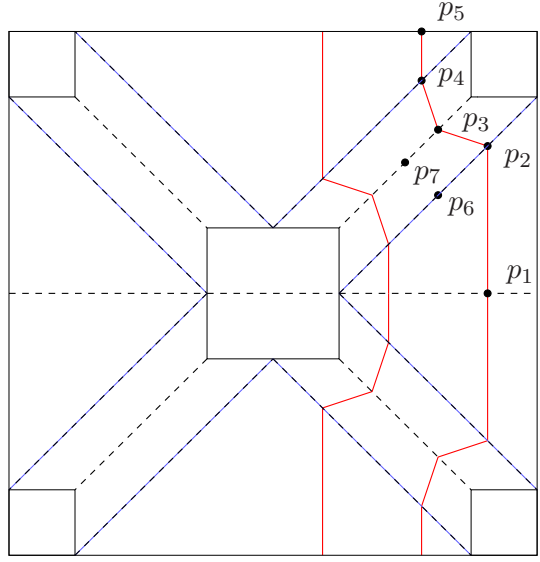}
\caption{In red a couple of fibers $l^m$.}
\label{fig5}
\end{figure}

\medskip

\noindent \textbf{Step 1.}
Let us define
\begin{equation*}
\tilde{M}_{\e_k}:=\bigl\{m\in M_{\e_k}\cap(-1/2,1/2) \colon l^m_{\e_k}\cap S_{w_k}=\emptyset\bigr\},
\end{equation*}
i.e., the set of the points $m$ whose fibers $l^m_{\e_k}$ do not intersect $S_{w_k}$. We will show in this step that
$\tilde{M}_{\e_k}$ tends to vanish:
\begin{equation}\label{taglio completo}
\lim_{k\rightarrow\infty} \HH^1(\tilde{M}_{\e_k})=0.
\end{equation}
\emph{The key in this step is that along these fibers $w_k$ is regular, and therefore it cannot converge to $w$, 
since it is not regular}.
We prove \eqref{taglio completo} by contradiction assuming that there exists $\lambda>0$ and a subsequence 
not relabeled such that
\begin{equation}\label{bound1 step1}
\HH^1(\tilde{M}_{\e_k})\geq\lambda \;\text{ for any } k\in\N.
\end{equation}

First of all, we straighten the fibers.
Fixed $m\in M_{\e}$,  we define $\psi_{\e,m}$ as the unique isometry that transform $l^m_\e$ in the segment 
$\{m\}\times(-|l^m_\e|/2,|l^m_\e|/2)$ keeping the direction.
Then, we define $\varphi_{\e,m}:(-1/2,1/2)\rightarrow l^m_\e$ by setting 
$\varphi_{\e,m}(s):=\psi^{-1}_{\e,m}((m,|l^m_\e|s))$.
We also assume that $w_k$ coincides with its precise 
representative defined as in \cite[Remark 3.79 and Corollary 3.80]{Amb00}.
Then,  by \cite[Theorems 3.28, 3.107 and 3.108]{Amb00}, for a.e. $m\in\tilde{M}_{\e_k}$ the composition 
$\omega_{k,m}:=w_k\circ\varphi_{\e_k,m}$ is a Sobolev map and its derivative is given by
$\omega'_{k,m}=\partial_1 w_k(\varphi_{\e_k,m})'_1+\partial_2 w_k(\varphi_{\e_k,m})'_2$.
In particular, since $|\varphi_{\e_k,m}'|\leq|l^m_\e|\leq2$, we have
\begin{equation*}
\int_{\tilde{M}_{\e_k}}\int_{-1/2}^{1/2}|\omega_{k,m}'|^2\,\mathrm{d}s\,\mathrm{d}m\leq4\int_Q|\nabla w_k|^2\dx.
\end{equation*}
Up to a subsequence, there exists a set $W\subset Q$ of null measure such that $w_k\rightarrow w$ pointwise in $Q\setminus W$.
Fixed a $\lambda'>0$, we select for each $k\in\N$ a $m_k\in\tilde{M}_{\e_k}$ so that
$\omega_{k,m_k}$ is a Sobolev map, $\HH^1(l^{m_k}_{\e_k}\cap W)=0$ and 
\begin{equation}\label{bound2 step1}
\int_{-1/2}^{1/2}|\omega_{k,m_k}'|^2\,\mathrm{d}s-\lambda'\leq\inf_{m\in\tilde{M}_{\e_k}}\int_{-1/2}^{1/2}|\omega_{k,m}'|^2\,\mathrm{d}s.
\end{equation}
By \eqref{bound1 step1}-\eqref{bound2 step1} the sequence $(\omega_{k,m_k}')$ is bounded in $L^2((-1/2,1/2),\R^2)$.
On the other hand, since $w_k\rightarrow w$ pointwise in $Q\setminus W$, $(\omega_{k,m_k})$ is also converging
pointwise a.e. to a function with discontinuity set $\{0\}$ and this is a contradiction. Therefore
\eqref{taglio completo} has to hold true.

\medskip

\noindent \textbf{Step 2.}
As we already said, in order to cut the bundle of fibers, the best choice is to make the cut in $T_\e$, and more precisely
along the set $R_\e$ as defined in \eqref{nervatura}. Indeed, here the hard region $\e P$ is thin just $1/\sqrt{2}$.
On the other hand, outside $T_\e$ the best choice is to make the cut along the diagonal part of the boundary of $T_\e$ itself.
Indeed, here the hard region $\e P$ is thin $(1+4\varrho)/\sqrt{2}$ (that it is smaller than $3/4$, since $\varrho<1/7$).
\emph{The key in this step is the fact that the ratio of the costs between the optimal cuts outside and inside $T_\e$ is $1+4\varrho$}.

Let us define 
\begin{equation*}
M^1_{\e_k}:=\bigl\{m\in M_{\e_k}\cap(-1/2,1/2) \colon l^m_{\e_k}\cap S_{w_k}\cap T_{\e_k}\neq\emptyset\bigr\},
\end{equation*}
i.e., the set of the points $m$ whose fibers $l^m_{\e_k}$ intersect $S_{w_k}$ in $T_{\e_k}$, and 
$M^2_{\e_k}:=(M_{\e_k}\cap(-1/2,1/2))\setminus(M^1_{\e_k}\cup\tilde{M}_{\e_k})$.
Note that by the previous step
\begin{equation}\label{taglio completo2}
\lim_{k\rightarrow\infty} \HH^1(M^1_{\e_k})+\HH^1(M^2_{\e_k})=\frac{3}{4}.
\end{equation}
The bundle of the fibers $\{l^m_\e \colon m\in M_\e\cap(-1/2,1/2)\}$ has cross section $3/4$ in $M_\e\times\{0\}$,
$1/\sqrt{2}$ in $\e\bigcup_{i\in\Z}((R^1\cup R^2)+i)$, and $(1+4\varrho)/\sqrt{2}$ in 
$\e\bigcup_{i\in\Z}((\hat{R}^1\cup \hat{R}^2)+i)$, where $\hat{R}^1$ is the segment of endpoints $(1/8,1/8-\varrho)$ 
and $(3/8+\varrho,3/8)$, and $\hat{R}^2$ is the reflection of $\hat{R}^1$ with respect to the axis $\{x_1=0\}$.
Therefore, if we first project $S_{w_k}\cap T_{\e_k}$ along the fibers on $\e_k\bigcup_{i\in\Z}((R^1\cup R^2)+i)$, 
and then back to $M_{\e_k}\times\{0\}$, we get the estimate
\begin{equation}\label{dentro}
\HH^1(S_{w_k}\cap T_{\e_k})\geq\frac{2\sqrt{2}}{3}\HH^1(M^1_{\e_k}),
\end{equation}
while if we first project $S_{w_k}\setminus T_{\e_k}$ along the fibers on $\e_k\bigcup_{i\in\Z}((\hat{R}^1\cup \hat{R}^2)+i)$, 
and then back to $M_{\e_k}\times\{0\}$, we get the estimate
\begin{equation}\label{fuori}
\HH^1(S_{w_k}\setminus T_{\e_k})\geq(1+4\varrho)\frac{2\sqrt{2}}{3}\HH^1(M^2_{\e_k}).
\end{equation}

We now prove that 
\begin{equation}\label{main estimate plus}
\liminf_{k\rightarrow+\infty}\HH^1(M^1_{\e_k})\geq\frac{3}{4}-\frac{3\eta}{8\sqrt{2}\varrho}.
\end{equation}
This, together with \eqref{dentro} will provide \eqref{main estimate}. We proceed by contradiction 
assuming that \eqref{main estimate plus} is false. Thanks to \eqref{taglio completo2}, this is equivalent to assume 
\begin{equation*}\begin{split}
\liminf_{k\rightarrow+\infty}\HH^1(M^2_{\e_k})
\geq\frac{3\eta}{8\sqrt{2}\varrho}.
\end{split}\end{equation*}
Then, by using \eqref{dentro}-\eqref{fuori} and again \eqref{taglio completo2},
\begin{equation*}\begin{split}
F&(w,Q)\geq\liminf_{k\rightarrow+\infty}\HH^1(S_{w_k})\\
&=\liminf_{k\rightarrow+\infty}\HH^1(S_{w_k}\cap T_{\e_k})+\HH^1(S_{w_k}\setminus T_{\e_k})\\
&\geq\frac{2\sqrt{2}}{3}\liminf_{k\rightarrow+\infty}\HH^1(M^1_{\e_k})
+(1+4\varrho)\frac{2\sqrt{2}}{3}\liminf_{k\rightarrow+\infty}\HH^1(M^2_{\e_k})\\
&=\frac{1}{\sqrt{2}}+4\varrho\frac{2\sqrt{2}}{3}\liminf_{k\rightarrow+\infty}\HH^1(M^2_{\e_k})
\geq\frac{1}{\sqrt{2}}+\eta.
\end{split}\end{equation*}
On the other hand, by \eqref{vicino} we have, since $w$ is a solution to \eqref{min prob}, 
$F(w,Q)<1/\sqrt{2}+\eta$, thus a contradiction.
\end{proof}

\begin{figure}
\centering
\psfrag{1}{$p$}
\psfrag{2}{$3\varrho$}
\psfrag{3}{$x_2=x_1+\varrho$}
\includegraphics[width=7.5cm]{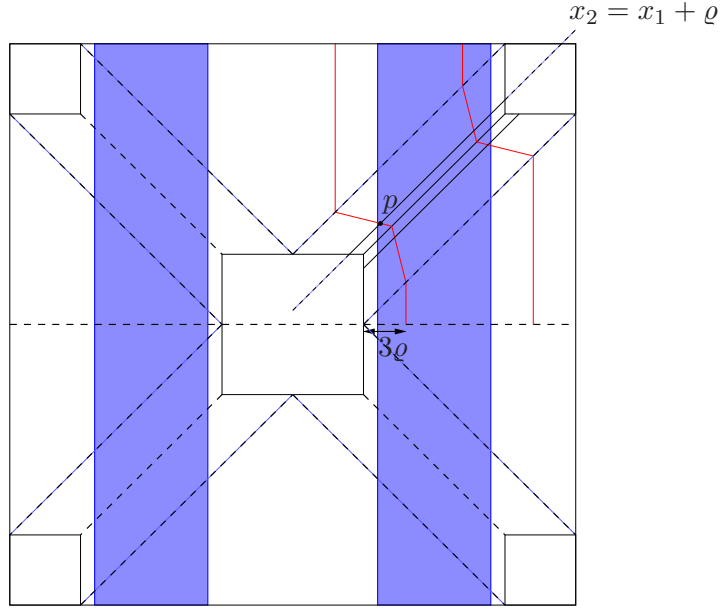}
\caption{In blue the set $U$.}
\label{fig6}
\end{figure}

\begin{remark}
Note that while \eqref{main estimate} just says that the discontinuity set is mainly localized in $T_{\e_k}$,
\eqref{main estimate plus} is stronger and it says that the discontinuity set is spread so to cut the fibers.

In particular, consider the set $U:=((-1/2+\varrho,-1/8-\varrho)\cup(1/8+\varrho,1/2-\varrho))\times\R$ and 
\begin{equation*}
U_\e:=\e\bigcup_{i\in\Z}(U+(i,0)).
\end{equation*}
The fiber $l^m$ intersects the straight line $\{x_2=x_1+\varrho\}$ (see Figure \ref{fig6}) at the point 
\begin{equation*}
p:=(1/24+2m/3+8m\varrho/3-4\varrho/3,1/24+2m/3+8m\varrho/3-\varrho/3).
\end{equation*}
Therefore, if $1/8+3\varrho\leq|m|\leq1/2-3\varrho$, then $l^m\cap T\subset U$.
Let us define 
\begin{equation*}
\hat{M}^1_{\e_k}:=\bigl\{m\in M^1_{\e_k} \colon 1/8+3\varrho\leq\tfrac{m}{\e_k}-[\tfrac{m}{\e_k}]\leq1/2-3\varrho\bigr\}.
\end{equation*}
The set $\hat{M}^1_{\e_k}$ is constituted by points $m$ whose fiber $l^m_{\e_k}$ intersect $S_{w_k}$ in $T_{\e_k}\cap U_{\e_k}$,
i.e., $l^m_{\e_k}\cap S_{w_k}\cap T_{\e_k}\cap U_{\e_k}\neq\emptyset$. Arguing as for \eqref{dentro} we get 
\begin{equation*}
\HH^1(S_{w_k}\cap T_{\e_k}\cap U_{\e_k})\geq\frac{2\sqrt{2}}{3}\HH^1(\hat{M}^1_{\e_k}).
\end{equation*}
By \eqref{main estimate plus} we have also
\begin{equation*}
\liminf_{k\rightarrow+\infty}\HH^1(\hat{M}^1_{\e_k})\geq\frac{3}{4}-\frac{3\eta}{8\sqrt{2}\varrho}-6\varrho
\end{equation*}
and then 
\begin{equation}\label{main estimate colonne}
\liminf_{k\rightarrow+\infty}\HH^1(S_{w_k}\cap T_{\e_k}\cap U_{\e_k})\geq\frac{1}{\sqrt{2}}-\frac{\eta}{4\varrho}-4\sqrt{2}\varrho.
\end{equation}
\end{remark}

\bigskip

As we observed before, when $t$ is larger than a threshold $t_0$ (depending on $f$, $g$ and $\delta$), the solutions
to the problem \eqref{min prob}-\eqref{constant} have the form $w=t\chi_{(-1/2,1/2)\times(\hat{x}_2,1/2)}$ for some 
$\hat{x}_2\in(1/2,1/2)$. Let us give for these solutions a complementary estimate to \eqref{main estimate colonne}, 
in the sets $Q\setminus U_{\e_k}$ and when $t$ is large. The proof is based on \cite[Theorem 2]{BLZ}.
The point is that for $t$ large at the microscopic level it is energetically convenient to break also
the soft inclusions, instead to stretch them.

\begin{theorem}
let $w:=t\chi_{(-1/2,1/2)\times(\hat{x}_2,1/2)}$ and $(w_k)$ be a sequence in $SBV^2(Q)$ converging to $v$ in $L^1(Q)$. 
Given $\varrho>0$, there exists a threshold $t_0=t_0(\varrho)$ such that, if $t\geq t_0$, then 
\begin{equation}\label{second main estimate}
\liminf_{k\rightarrow+\infty}F_{\e_k}(w_k,Q\setminus U_{\e_k})\geq\frac{1}{2}+4\varrho.
\end{equation}
\end{theorem}

\begin{figure}
\centering
\includegraphics[width=11cm]{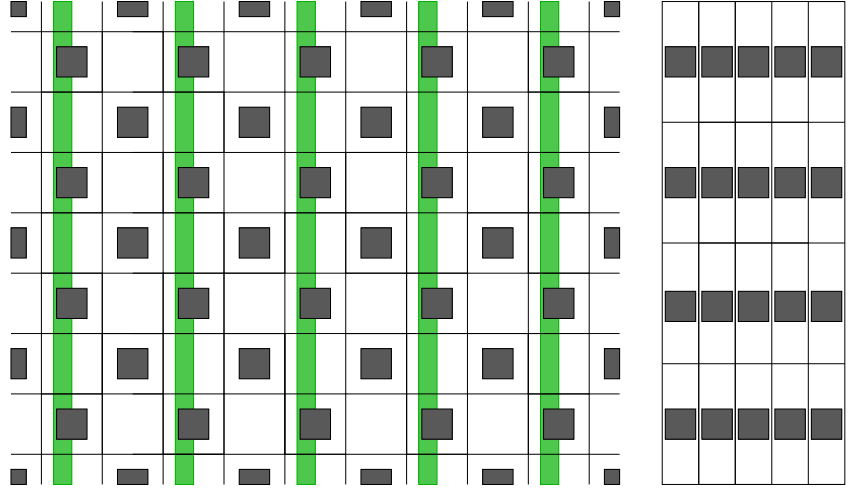}
\caption{In green the strips $S_k^i$. In gray and on the right the set $\R^2\setminus\tilde{P}$.}
\label{fig7}
\end{figure}

\begin{proof}
Let $S_k^i$ be the open strip $[(-\e_k(1/8+\varrho),0)\times(-1/2,1/2)]+(\e_k i,0)$, and $S_k$ the union
of the strips $S_k^i$ included in $Q$ (see Figure \ref{fig7}). Moreover, let $\lambda>0$ be a small quantity and 
let $(\lambda_k)$ be a sequence such that 
\begin{equation*}
\lambda_k\rightarrow+\infty\quad  \text{and}\quad \sup_{k}\Big(\lambda_k\int_Q|w_k-w|\dx\Big)\leq\lambda.
\end{equation*}
Let $S_k^{i_k}\in \Z$ be a solution to
\begin{equation*}
\min\Bigl\{F_{\e_k}(w_k,S_k^i)+\lambda_k\int_{S_k^i}|w_k-w|\dx \colon S_k^i\subset Q\Bigr\}.
\end{equation*}
We first restrict $w_k$ to $S_k^{i_k}$, and then we extend it to the strip 
$[(-\e_k(1/8+\varrho),\e_k(1/8+\varrho))\times(-1/2,1/2)]+(\e_k i,0)$
by reflection with respect to the axis $x_1=i_k \e_k$; we denote by $\tilde{w}_k$ such an extension. 
Then we extend further $\tilde{w}_k$ by periodicity in the $x_1$-variable to the whole $\R\times(-1/2,1/2)$,
with period $\e_k(1/4+2\varrho)$. 
The penalization term ensures that $\tilde{w}_k\to w$ in $L^1(Q)$. 

Consider now the sets $\Om:=(-1/8-\varrho,1/8+\varrho)\times(-1/2,1/2)$,
\begin{equation*}\begin{split}
\tilde{D}&:=[-\tfrac{1}{8},\tfrac{1}{8}]\times[-\tfrac{1}{8},\tfrac{1}{8}],\\
\tilde{P}&:=\R^2 \setminus \bigcup_{(i,j)\in\Z^2} \bigl(\tilde{D}+\bigl(\bigl(\tfrac{1}{4}+2\varrho\bigr)i,j\bigr)\bigr)
\end{split}\end{equation*}
(see Figure \ref{fig7}), and the functional $\tilde{F}_\e:L^1(\Omega)\to[0,+\infty]$ defined as
\begin{equation*}
\tilde{F}_\e(u,\Om):=
\begin{cases}
\displaystyle{\int_{\Om \cap \e \tilde{P}}|\nabla u|^2\dx+\e\int_{\Om\setminus\e \tilde{P}}|\nabla u|^2\dx+\HH^1(S_u)} 
& \text{if }\ u\in SBV^2(\Om), \cr
+\infty & \text{otherwise in } L^1(\Om).
\end{cases}
\end{equation*}
Note that $\tilde{P}$ is connected. By construction of the sequence $(\tilde{w}_k)$ we have 
\begin{equation}\label{teor3 stima1}
F_{\e_k}(w_k,S_k)+\lambda_k\int_{S_k}|w_k-w|\dx \geq\frac{1}{2}\tilde{F}_{\e_k}(\tilde{w}_k,\Om),
\end{equation}
while by \cite[Theorem 2]{BLZ}, with some slight modifications due to the different cell of periodicity and the different 
size of the soft inclusions, for $t$ larger than a threshold $t_0=t_0(\varrho)$
\begin{equation}\label{teor3 stima2}
\liminf_{k\rightarrow+\infty}\tilde{F}_{\e_k}(\tilde{w}_k,\Om)\geq\frac{1}{4}+2\varrho.
\end{equation}
By \eqref{teor3 stima1} and \eqref{teor3 stima2}, and being $\lambda$ arbitrary, we get 
\begin{equation*}
\liminf_{k\rightarrow+\infty}F_{\e_k}(w_k,S_k)\geq\frac{1}{8}+\varrho.
\end{equation*}
By repeating a similar estimate on the remaining part of $Q\setminus U_{\e_k}$, we have the full 
estimate~\eqref{second main estimate}.
\end{proof}

\section{Conclusions}

\noindent
Let us summarize estimates \eqref{main estimate colonne} and \eqref{second main estimate} in a comprehensive result.
\begin{mtheorem}
Given $\varrho\in(0,1/7)$, choose $t>0$ small enough so that \eqref{vicino} is verified for $\eta=4\varrho^2$, and $\delta>0$ 
small enough so that the solutions to the problem \eqref{min prob}-\eqref{constant} are not affine and discontinuous. 
Let $w$ be one of this solutions, with discontinuity set $(-1/2,1/2)\times\{\hat{x}_2\}$ for a certain $\hat{x}_2\in[-\delta/2,\delta/2]$, 
and $(w_k)$ one of its recovery sequence. Moreover, let $v:=t'\chi_{(-1/2,1/2)\times(\hat{x}_2,1/2)}$ and $(v_k)$ be a sequence
in $SBV^2(Q)$ converging to $v$ in $L^1(Q)$. If $t'>t$ is large enough and
\begin{equation}\label{irreversibile}
S_{v_k}\supset S_{w_k},
\end{equation}
then
\begin{equation}\label{finale}
\liminf_{k\rightarrow+\infty}F_{\e_k}(v_k,Q)\geq\frac{1}{2}+\frac{1}{\sqrt{2}}+(3-4\sqrt{2})\varrho.
\end{equation}
\end{mtheorem}

\bigskip

Since $\varrho$ can be taken arbitrarily small, \eqref{finale} show that the toughness of the material increases
from one to $1/2+1/\sqrt{2}$.
From the physical point of view, the explanation is that the bridging of the soft inclusions, being energetically 
favorable when the opening of the macroscopic crack is small, originates a deflection of the crack path with respect 
to the straight one. 
Because of the irreversibility of the crack process, this deflection persists also when the 
opening of the crack is large and a straight path should be energetically favorable with respect to the deflected one.

This behavior cannot be captured by the $\Gamma$-limit $F$, since it is obtained by a minimization problem 
at microscopic level for any fixed opening of the crack. A general effective model that takes into account the 
irreversibility of the crack process at microscopic level (i.e., condition \eqref{irreversibile}) should provide
accordingly to \eqref{finale} an effective surface energy density $g_{\eff}$ such that 
\begin{equation*}
g_{\eff}(t,e)=\frac{1}{2}+\frac{1}{\sqrt{2}}>1=g(t,e)
\end{equation*}
for $t$ large enough.
Note also that our result shows that for $(F_\e)$ homogenization and quasistatic evolution of cracks do not commute.
This is in contrast with what happens in the case of a family of functionals that do not depend on the opening of the crack,
but have standard growth conditions: not only the $\Gamma$-limit still does not depend on the opening, but homogenization 
and evolution commute (see \cite{Ponsimini}).

To conclude, some generalizations must be envisaged in order to combine $\Gamma$-convergence of energies and 
irreversibility of the crack process at microscopic level. However, this seems to be a challenging problem
at the moment.

\vspace{10pt}
\centerline{\textsc{\large{Acknowledgments}}}

\vspace{4pt}
\noindent
The author gratefully thanks Giuliano Lazzaroni and Caterina Ida Zeppieri for stimulating discussions.
This work has been supported by the ERC Advanced Grant 290888--\emph{QuaDynEvoPro}.

\vspace{4pt}


\end{document}